\documentclass[10pt,leqno]{amsart}
\usepackage{indentfirst, amssymb,amsmath,amsthm}
\newtheorem*{theoA}{Theorem A}
\newtheorem*{theoB}{Theorem B}
\newtheorem*{theoC}{Theorem C}
\newtheorem*{theoD}{Theorem D}
\newtheorem*{theoE}{Theorem E}

\newtheorem*{conjA}{Conjecture A}
\newtheorem*{conB}{Conjecture B}
\newtheorem*{cor A}{Corollary A}
\newtheorem*{cor B}{Corollary B}
\newtheorem{theo}{Theorem}[section]
\newtheorem{lem}{Lemma}[section]

\newtheorem{note}{Note}[section]
\newtheorem{exm}{Example}[section]
\newtheorem{defi}{Definition}[section]
\newtheorem{rem}{Remark}[section]

\newcommand{\ol}{\overline}
\newcommand{\be}{\begin{equation}}
\newcommand{\ee}{\end{equation}}
\newcommand{\beas}{\begin{eqnarray*}}
\newcommand{\eeas}{\end{eqnarray*}}
\newcommand{\bea}{\begin{eqnarray}}
\newcommand{\eea}{\end{eqnarray}}

\numberwithin{equation}{section}
\begin{document}
\title[On a conjecture of L\"u, Li and Yang]{On a conjecture of L\"u, Li and Yang}
\date{}
\author[I. Lahiri and S. Majumder]{Indrajit Lahiri and Sujoy Majumder}
\address{Department of Mathematics, University of Kalyani, Kalyani, West Bengal 741235, India.}
\email{ilahiri@hotmail.com}
\address{Department of Mathematics, Raiganj University, Raiganj, West Bengal 733134, India.}
\email{sm05math@gmail.com, sujoy.katwa@gmail.com}
\renewcommand{\thefootnote}{}
\renewcommand{\thefootnote}{\arabic{footnote}}
\setcounter{footnote}{0}
\begin{abstract} 
In connection to a conjecture of W. L\"u. Q. Li and C. Yang we prove a result on small function sharing by a power of a meromorphic function with few poles and 
its derivative. Our results improve a number of known results.
\end{abstract}
\thanks{Typeset by \AmS -\LaTeX}
\maketitle
\section{Introduction Definitions and Results}
In the paper a meromorphic function means it is meromorphic in the open complex plane $\mathbb{C}$. we use the standard notations of Nevanlinna theory e.g., 
$N(r, f)$, $m(r, f)$, $T(r, f)$, N(r, a; f)$, \ol N(r, a; f)$, $m(r, a;f)$ etc.\{see \cite{4}\}. We denote by $S(r, f)$ a quantity, not necessarily the same at 
each of its occurrence, that satisfies the condition $S(r, f) = o\{T(r, f)\}$ as $r \to \infty$ except possibly a set of finite linear measure.

A meromorphic function $a = a(z)$ is called a small function of a meromorphic function $f$, if $T(r, a) = S(r, f)$. Let us denote by $S(f)$ the class of all 
small functions of $f$. Clearly $\mathbb{C} \subset S(f)$ and if $f$ is a transcendental function, then every polynomial is a member of $S(f)$.

Let $f$ and $g$ be two non-constant meromorphic functions and $a \in S(f) \cap S(g)$. If $f - a$ and $g - a$ have the same zeros with the same multiplicities, 
then we say that $f$ and $g$ share the small function $a$ CM (counting multiplicities) and if we do not consider the multiplicities, then we say that $f$ and $g$ 
share the small function $a$ IM (ignoring multiplicities).

Let $k$ be a positive integer and $a\in S(f)$. We use $N_{k)}(r,a;f)$ to denote the counting function of zeros of $f - a$ with  multiplicity not greater than $k$
, $N_{(k+1}(r,a;f)$ to denote the counting function of zeros of $f - a$ with multiplicity greater than $k$. Similarly we use  $\ol N_{k)}(r,a;f)$ and $\ol 
N_{(k+1}(r,a;f)$ are their respective reduced functions.

In 1996, Br\"{u}ck \cite{2} studied the relation between $f$ and $f'$ if an entire function $f$ shares only one finite value CM with it's derivative $f'$.
In this direction an interesting conjecture was proposed by Br\"{u}ck \cite{2}, which is still open in its full generality.
\begin{conjA} Let $f$ be a non-constant entire function. Suppose \beas \rho_{1}(f):=\limsup\limits_{r \to \infty}\frac{\log \log T(r,f)}{\log r},\eeas the 
hyper-order of $f$, is not a positive integer or infinity. If $f$ and $f'$ share a finite value $a$ CM, then \bea\label{s} \frac{f'-a}{f-a}=c\eea for some 
non-zero constant $c$.\end{conjA}
The Conjecture for the special cases $(1)$ $a=0$ and $(2)$ $N(r,0;f')=S(r,f)$ had been established by  Br\"{u}ck \cite{2}.
From the differential equations \bea\label{1} \frac{f'-a}{f-a}=e^{z^{n}},\;\;\; \frac{f'-a}{f-a}=e^{e^{z}},\eea we see that when $\rho_{1}(f)$ is a positive 
integer or infinity, the conjecture does not hold.\par
The conjecture for the case that $f$ is of finite order had been proved by Gundersen and Yang \cite{4aa}, the case that $f$ is of infinite order with 
$\rho_{1}(f)<\frac{1}{2}$ had been proved by Chen and Shon \cite{2a}. Recently Cao \cite{1} proved that the Br\"{u}ck conjecture is also true when $f$ is of 
infinite order with $\rho_{1}(f)=\frac{1}{2}$. But the case $\rho_{1}(f)> \frac{1}{2}$ is still open.
However, the corresponding conjecture for meromorphic functions fails in general (see \cite{4aa}). For example, if \beas f(z)=\frac{2e^{z}+z+1}{e^{z}+1},\eeas 
then $f$ and $f'$ share $1$ CM, but (\ref{s}) does not hold.\par
It is interesting to ask what happens if $f$ is replaced by a power of it, say, $f^{n}$ in Br\"{u}ck's conjecture. From (\ref{1}) we see that the conjecture does 
not hold without any restriction on the hyper-order when $n=1$. So we only need to focus on the problem when $n\geq 2$.\par
Perhaps Yang and Zhang \cite{16a} were the first to consider the uniqueness of a power of an entire function $F=f^{n}$ and its derivative $F'$ when they share certain value and that leads to a specific form of the function $f$.
\par Yang and Zhang \cite{16a} proved that the Br\"{u}ck conjecture holds for the function $f^{n}$ and the order restriction on $f$ is not needed if $n$ is relatively large. Actually they proved the following result.
\begin{theoA}\cite{16a} Let $f$ be a non-constant entire function, $n(\geq 7)$ be an integer and let $F=f^{n}$. If $F$ and $F'$ share $1$ CM, then $F\equiv F'$, and $f$ assumes the form $f(z)=ce^{\frac{1}{n}z}$, where $c$ is a non-zero constant.\end{theoA}

Improving all the results obtained in \cite{16a}, Zhang \cite{16aa} proved the following theorem. 
\begin{theoB}\cite{16aa} Let $f$ be a non-constant entire function, $n$, $k$ be positive integers and $a(\not\equiv 0,\infty)$ be a meromorphic small function of $f$. If $f^{n}-a$ and $(f^{n})^{(k)}-a$ share $0$ CM and $n\geq k+5$, then $f^{n}\equiv (f^{n})^{(k)}$, and $f$ assumes the form $f(z)=ce^{\frac{\lambda}{n}z}$, where $c$ is a non-zero constant and $\lambda^{k}=1$.\end{theoB}
In 2009, Zhang and Yang \cite{17} further improved the above result in the following manner.
\begin{theoC}\cite{17} Let $f$ be a non-constant entire function, $n$, $k$ be positive integers and $a(\not\equiv 0,\infty)$ be a meromorphic small function of $f$. Suppose $f^{n}-a$ and $(f^{n})^{(k)}-a$ share $0$ CM and $n\geq k+2$. Then conclusion of \textrm{Theorem B} holds.\end{theoC}
In 2010, Zhang and Yang \cite{18} further improved the above result in the following manner.
\begin{theoD}\cite{18} Let $f$ be a non-constant entire function, $n$ and $k$ be positive integers. Suppose $f^{n}$ and $(f^{n})^{(k)}$ share $1$ CM and $n\geq k+1$. Then conclusion of \textrm{Theorem B} holds.\end{theoD}
In 2011, L\"{u} and Yi \cite{11} proved the following extension of Theorem D.
\begin{theoE} \cite{11} Let $f$ be a transcendental entire function, $n$, $k$ be two integers with $n\geq k+1$, $F=f^{n}$ and $Q\not\equiv 0$ be a polynomial. If $F-Q$ and $F^{(k)}-Q$ share $0$ CM, then $F\equiv F^{(k)}$ and $f(z)=ce^{wz/n}$, where $c$ and $w$ are non-zero constants such that $w^{k}=1$.\end{theoE}

\begin{rem} It is easy to see that the condition $n\geq k+1$ in \textrm{Theorem E} is sharp by the following example.\end{rem}
\begin{exm} Let $f(z)=e^{e^{z}}\int\limits_{0}^{z} e^{-e^{t}}(1-e^{t})t\; dt$ and $n=1$, $k=1$. Then \beas \frac{f'(z)-z}{f(z)-z}=e^{z}\eeas and $f'(z)-z$ and 
$f(z)-z$ share $0$ CM, but $f'\not\equiv f$. \end{exm}

In \cite{LLY} W. L\"u, Q. Li and C. Yang asked the question of considering two shared polynomials in Theorem E instead of a single shared polynomial. They 
answered the question for the first derivative of the power of a transcendental entire function and further proposed the following conjecture:
\begin{conB} Let $f$ be a transcendental entire function, $n$ be a positive integer. If $f^{n} - Q_{1}$ and $(f^{n})^{(k)} - Q_{2}$ share $0$ CM and $n \geq k + 
1$, then $\displaystyle (f^{n})^{(k)} = \frac{Q_{2}}{Q_{1}}f^{n}$, where $Q_{1}$ and $Q_{2}$ are polynomials with $Q_{1}Q_{2} \not\equiv 0$. If, further, $Q_{1} 
\equiv Q_{2}$, then $\displaystyle f = ce^{\frac{\omega z}{n}}$, where $c$ and $\omega$ are nonzero constants such that $\omega^{k} = 1$. \end{conB}

Recently the second author \cite{SM} fully resolved {\bf Conjecture B}. Thus giving rise to a further investigation of the possibility of replacing in {\bf 
Conjecture B}  the shared polynomials by shared small functions. In the paper we, in one hand solve this problem and also in the other hand we try to relax the 
nature sharing of small functions, thereby improve a number of known results including that in \cite{SM}.

Extending the idea of weighted sharing \{\cite{IL1, IL2}\},  Lin and Lin \cite{7a} introduced  the notion of weakly weighted sharing which is defined as follows.
\begin{defi}\cite{7a} Let $f$ and $g$ be two non-constant meromorphic functions sharing a ``IM", for $a\in S(f)\cap S(g)$, and  $k$  be a positive integer or 
$\infty$.\begin{enumerate}\item[(i)] $\ol N_{k)}^{E}(r,a)$ denotes the counting function of those zeros of $f - a$ whose multiplicities are equal to the 
corresponding zeros of $g - a$, both of their multiplicities are not greater than $k$, where each zero is counted only once. \item[(ii)] $\ol N_{(k}^{0}(r,a)$ 
denotes the reduced counting function of those zeros of $f - a$ which are zeros of $g - a$, both of their multiplicities are not less than $k$, where each zero 
is counted only once.\end{enumerate}  \end{defi}
\begin{defi}\cite{7a} For $a\in S(f) \cap S(g)$, if $k$ is a positive integer or $\infty$ and \beas \ol N_{k)}(r,a;f)-\ol N_{k)}^{E}(r,a)=S(r,f),\;\;\;\ol 
N_{k)}(r,a;g)-\ol N_{k)}^{E}(r,a)=S(r,g);\eeas \beas \ol N_{(k+1}(r,a;f)- \ol N_{(k+1}^{0}(r,a)=S(r,f),\;\;\;\ol N_{(k+1}(r,a;g)- \ol N_{(k+1}^{0}(r,a)=S(r,g);
\eeas
or if $k=0$ and \beas \ol N(r,a;f)-\ol N_{0}(r,a)=S(r,f),\;\;\;\ol N(r,a;g)-\ol N_{0}(r,a)=S(r,g),\eeas
then we say $f$ and $g$  weakly share $a$ with weight $k$. Here we write $f$, $g$ share $``(a,k)"$ to mean that $f$, $g$ weekly share $a$ with weight $k$.\end{defi}
Obviously, if $f$ and $g$ share $``(a,k)"$, then $f$ and $g$ share $``(a,p)"$ for any $p\;\;(0\leq p\leq k)$. Also we note that $f$ and $g$ share $a$ $``IM"$ or 
$``CM"$ if and only if $f$ and $g$ share $``(a,0)"$ or $``(a,\infty)"$, respectively (for the definitions of $``IM"$ and $``CM"$ see pp. 225 - 226 \cite{12a}). \par

We note that a rational function $f$ with $\ol N(r, \infty ; f) = S(r, f)$ must be a polynomial. Also a small function of a polynomial must be a constant. Since $k \geq 1$, clearly if $f$ is a polynomial, then the relation $(f^{n})^{(k)} = cf^{n}$ does not hold for any nonzero constant $c$ and $n \geq k$. Therefore in the following theorems we assume $f$ to be transcendental.
\begin{theo}\label{t1} 
Let $f$ be a transcendental meromorphic function such that $ N(r,\infty;f)=S(r,f)$ and $a_{i} = a_{i}(z)(\not\equiv 0,\infty)$ be small functions of $f$, where $i=1,2$. Let $n$ and $k$ be two positive integers such that $n\geq k+1$. If $f^{n}-a_{1}$ and $(f^{n})^{(k)}-a_{2}$ share $``(0,1)"$, then $(f^{n})^{(k)}\equiv \frac{a_{2}}{a_{1}}f^{n}$. Furthermore, if $a_{1} \equiv a_{2}$, then $f(z)=ce^{\frac{\lambda}{n}z}$ where $c$ and $\lambda$ are non-zero constants such that $\lambda^{k}=1$.
\end{theo}
\begin{theo}\label{t2} 
Let $f$ be a transcendental meromorphic function such that $\ol  N(r,\infty;f)=S(r,f)$ and $a_{i} = a_{i}(z)(\not\equiv 0,\infty)$ be small functions of $f$, where $i=1,2$. 
Let $n$ and $k$ be two positive integers such that $n\geq k$. If $f^{n}-a_{1}$ and $(f^{n})^{(k)}-a_{2}$ share $``(0,0)"$ and $\ol N_{2)}(r,0;f)=S(r,f)$, then 
$(f^{n})^{(k)}\equiv \frac{a_{2}}{a_{1}}f^{n}$. Furthermore, if $a_{1} \equiv a_{2}$, then $f^{n}\equiv (f^{n})^{(k)}$ and $f$ assumes the form 
$f(z)=ce^{\frac{\lambda}{n}z}$, where $c$ is a non-zero constant and $\lambda^{k}=1$.
\end{theo}
\begin{note}\label{n1} If $k \geq 2$, then in Theorem \ref{t2} instead of $\ol N_{2)}(r,0;f)=S(r,f)$ we can assume $N_{1)}(r,0;f)=S(r,f)$.\end{note}
\begin{rem} It is easy to see that the condition $n\geq k+1$ in \textrm{Theorem \ref{t1}} is sharp by the following examples.\end{rem}
\begin{exm} Let $f(z)=e^{2z}+z$. Then $f-a_{1}$ and $f'-a_{2}$ share $0$ CM and $N(r,\infty;f)=0$, but $f'\not\equiv \frac{a_{2}}{a_{1}}f$, where $a_{1}(z)=z+1$ 
and $a_{2}(z)=3$. \end{exm}
\begin{exm} Let $f(z)=e^{2z}+z^{2}+z$. Then $f-a_{1}$ and $f'-a_{2}$ share $0$ CM and $N(r,\infty;f)=0$, but $f'\not\equiv \frac{a_{2}}{a_{1}}f$, where 
$a_{1}(z)=z^{2}+z+1$ and $a_{2}(z)=2z+3$. \end{exm}
\begin{exm} Let \beas f(z)=e^{e^{z^{2}}}+1,\;\;a_{1}(z)=\frac{1}{1+e^{-z^{2}}},\;\;a_{2}(z)=-\frac{2z}{1+e^{-z^{2}}}.\eeas 
	
We note that
\beas f(z)-a_{1}(z)=\frac{1}{e^{z^{2}}+1}\Big[\Big(e^{z^{2}}+1\Big)e^{e^{z^{2}}}+1\Big]\eeas and
\beas f'(z)-a_{2}(z)=\frac{2z}{1+e^{-z^{2}}}\Big[\Big(e^{z^{2}}+1\Big)e^{e^{z^{2}}}+1\Big].\eeas
 Then $f-a_{1}$ and $f'-a_{2}$ share $``(0,\infty)"$ and $ N(r,\infty;f)=0$, but $f\not\equiv \frac{a_{2}}{a_{1}}f'$. \end{exm}
\begin{exm} Let \beas f(z)=1-5(z+1)+ze^{z}\eeas and $a_{1}(z)=a_{2}(z)=-(4 + 4z + 5z^{2})$.
We note that \beas f(z)-a_{1}(z)=z(e^{z}+5z-1)\eeas and \beas f'(z)-a_{2}(z)=(z+1)(e^{z}+5z-1).\eeas Then $f-a_{1}$ and $f'-a_{2}$ share $``(0,\infty)"$ and $N(r,\infty;f)=0$, but $f\not\equiv f'$.
 \end{exm}
\begin{rem} It is easy to see that the conditions $\ol N_{2)}(r,0;f)=S(r,f)$ and $\ol N(r,\infty;f)=S(r,f)$ in \textrm{Theorem \ref{t2}} are essential by the following 
examples.\end{rem} 
\begin{exm} Let \beas f(z)=z^{2}+\frac{1}{2}e^{(z-1)^{2}},\;\;a_{1}(z)=z^{2}+\frac{1}{2}\;\; and\;\;a_{2}(z)=3z-1.\eeas

We note that \beas f(z)-(z^{2}+\frac{1}{2})=\frac{1}{2}\Big[e^{(z-1)^{2}}-1\Big]\eeas and \beas f'(z)-(3z-1)=(z-1)\Big[e^{(z-1)^{2}}-1\Big]. \eeas

Obviously $f-a_{1}$ and $f'-a_{2}$ share $0$ IM, and $\ol N_{2)}(r,0;f)\not=S(r,f)$ and $\ol N(r, \infty; f) = 0$, but $f'\not\equiv \frac{a_{2}}{a_{1}}f$. \end{exm}
\begin{exm} Let \beas f(z)=\frac{2}{1-e^{-2z}}.\eeas Clearly $f'(z)=-\frac{4e^{-2z}}{(1-e^{-2z})^{2}}$.

We note that \beas f(z)-1=\frac{1+e^{-2z}}{1-e^{-2z}}\;\;\; and\;\;\; f'(z)-1=-\frac{(1+e^{-2z})^{2}}{(1-e^{-2z})^{2}}. \eeas
Obviously $f$ and $f'$ share $1$ IM, $\ol N(r,\infty;f)\neq S(r,f)$ and $\ol N_{2)}(r,0;f) = 0$, but $f'\not\equiv f$. \end{exm}
 \begin{exm} Let $f(z) = 1 + \tan z$. Since $\tan z$ does not assume the values $\pm i$, it follows that $f(z)$ does not assume the values $1 \pm i$. So by the 
second fundamental theorem, $\ol N(r, 0; f) = \ol N_{2)}(r, 0; f) = T(r, f) + S(r, f)$ and $\ol N(r, \infty ; f) = T(r, f) + S(r, f)$. Also we see that $f'(z) - 1 
= (f(z) - 1)^{2}$ and so $f$ and $f'$ share the value $1$ IM, but $f \not\equiv f'$. \end{exm}

\section {Lemmas} In this section we present the lemmas which will be needed in the sequel.
\begin{lem}\label{l1}\cite{3} Suppose that $f$ is a transcendental meromorphic function and that \beas f^{n}(z)P(f(z))=Q(f(z)),\eeas where $P(f(z))$ and $Q(f(z))$
are differential polynomials in $f$ with functions of small proximity related to $f$ as the coefficients and the degree of $Q(f(z))$ is at most $n$. Then 
$m(r,P)=S(r,f).$ \end{lem}
\begin{lem}\label{l5} \cite{4} Let $f$ be a non-constant meromorphic function and let $a_{1}(z)$, $a_{2}(z)$ be two meromorphic functions such that 
$T(r,a_{i})=S(r,f)$, $i=1,2$. Then \beas T(r,f)\leq \ol N(r,\infty;f)+\ol N(r,a_{1};f)+\ol N(r,a_{2};f)+S(r,f).\eeas \end{lem}
\begin{lem}\label{l7}\cite{6a} Let $f(z)$ be a non-constant entire function and $k(\geq 2)$ be an integer. If $f(z)f^{(k)}(z)\not=0$, then $f(z)=e^{az+b}$, where 
$a\not=0, b$ are constant.\end{lem}

\begin{lem}\label{l9} Let $f$ be a non-constant meromorphic function such that $(f^{n})^{(k)}\equiv f^{n}$, where $k, n\in\mathbb{N}$. If $n\geq k$, then $f$ 
assumes the form $f(z)=ce^{\frac{\lambda}{n}z}$, where $c\in\mathbb{C}\setminus\{0\}$ and $\lambda^{k}=1$.
\end{lem}
\begin{proof} First we suppose \bea\label{r1} (f^{n})^{(k)}\equiv f^{n}.\eea
We claim that $f$ does not have any pole. In fact, if $z_{0}$ is a pole of $f$ with multiplicity $p$, then $z_{0}$ is a pole of $f^{n}$ with multiplicity $np$ 
and a pole of $(f^{n})^{(k)}$ with multiplicity $np+k$, which is impossible by (\ref{r1}). Hence $f$ is a non-constant entire function. From (\ref{r1}), it is 
clear that $f$ can not be a polynomial. Therefore $f$ is a transcendental entire function.\\
We now consider the following two cases.\\
{\bf Case 1.} Let $n>k$.\\
If $z_{1}$ is a zero of $f$ with multiplicity $q$, then $z_{1}$ is a zero of $f^{n}$ with multiplicity $nq$ and a zero of $(f^{n})^{(k)}$ with multiplicity $nq-k$
, which is impossible by (\ref{r1}). Therefore from (\ref{r1}), we conclude that $f^{n}(z)(f^{n}(z))^{(k)}\not=0$. If $k\geq 2$, then by \textrm{Lemma \ref{l7}} 
we have
$f(z)=ce^{\frac{\lambda}{n}z}$, where $c\in\mathbb{C}\setminus\{0\}$ and $\lambda^{k}=1$. Next we suppose $k=1$. Since $f(z)\not=0, \infty$, it follows that 
$f(z)=e^{\alpha(z)}$, where $\alpha(z)$ is a non-constant entire function. Now from (\ref{r1}) we have  $\alpha'(z)=\frac{1}{n}$, i.e., 
$\alpha(z)=\frac{1}{n}z+c_{0}$, where $c_{0}\in\mathbb{C}$. Consequently $f(z)=ce^{\frac{1}{n}z}$, where $c=e^{c_{0}}$.\\
{\bf Case 2.} Let $n=k$.\\
First we suppose $n=k=1$. Then from (\ref{r1}) we have $f(z)\equiv f'(z)$ and so $f(z)=ce^{z}$, where $c\in\mathbb{C}\setminus\{0\}$.\\
Next we suppose $n=k\geq 2$.
Let $F=f^{n}$. 

Then we have
\bea\label{r3} F^{(k)}&=&\frac{d^{k}}{dz^{k}}\Big\{f^{k}\Big\}\\&=& \frac{d^{k-1}}{dz^{k-1}}\Big\{kf^{k-1}f'\Big\}\nonumber\\&=& k\frac{d^{k-2}}{dz^{k-2}}\Big\{(k-1)f^{k-2}(f')^{2}+f^{k-1}f''\Big\}\nonumber\\&=& k(k-1)\frac{d^{k-2}}{dz^{k-2}}\Big\{f^{k-2}(f')^{2}\Big\} +k\frac{d^{k-2}}{dz^{k-2}}\Big\{f^{k-1}f''\Big\}\nonumber\\&=& k(k-1)\frac{d^{k-3}}{dz^{k-3}}\Big\{(k-2)f^{k-3}(f')^{3}\Big\}+k(k-1)\frac{d^{k-3}}{dz^{k-3}}\Big\{2f^{k-2}f'f''\big\}\nonumber\\&&+k\frac{d^{k-3}}{dz^{k-3}}\Big\{(k-1)f^{k-2}f'f''\Big\}+k\frac{d^{k-3}}{dz^{k-3}}\Big\{f^{k-1}f''\Big\}\nonumber\\&=&k(k-1)(k-2)\frac{d^{k-3}}{dz^{k-3}}\Big\{f^{k-3}(f')^{3}\Big\}+2k(k-1)\frac{d^{k-3}}{dz^{k-3}}\Big\{f^{k-2}f'f''\big\}\nonumber\\&&+k(k-1)\frac{d^{k-3}}{dz^{k-3}}\Big\{f^{k-2}f'f''\Big\}+k\frac{d^{k-3}}{dz^{k-3}}\Big\{f^{k-1}f''\Big\}\nonumber\\&=& \ldots\ldots\nonumber \\&=& k!(f')^{k}+R(f)\nonumber, \eea
where $R(f)$ is a differential polynomial in $f$ such that each term of $R(f)$ contains $f^{m}$ for some $m (1\leq m\leq n-1)$ as a factor. \par     
From (\ref{r1}), we observe that $f$ can not have any multiple zero. Let $z_{2}$ be a simple zero of $f$. Clearly $z_{2}$ is a zero of $F$ of multiplicity $k$. From (\ref{r1}), it is clear that $z_{2}$ is also a zero of $F^{(k)}$. On the other hand $z_{2}$ is a zero of $R(f)$. Now from (\ref{r3}), we observe that $z_{2}$ is a zero of $f'$, which is impossible. Therefore $f$ can not have any simple zero.
Hence $f$ does not have any zero. Since from (\ref{r1}) we see that $(f^{n}(z))^{(k)}f^{n}(z) \neq 0$, by Lemma \ref{l7} we have
$f(z)=ce^{\frac{\lambda}{n}z}$, where $c\in\mathbb{C}\setminus\{0\}$ and $\lambda^{k}=1$. This completes the proof.
\end{proof}

\section {Proofs of the theorems} 
\begin{proof}[Proof of Theorem \ref{t1}] Let \bea\label{e4} F=f^{n}.\eea
Since $S(r,f^{n})=S(r,f)$, from \textrm{Lemma \ref{l5}} we see that \beas nT(r,f)\leq \ol N(r,0;F)+\ol N(r,a_{1};F)+S(r,f^{n})=\ol N(r,0;f)+\ol N(r,a_{1};F)+S(r,f).\eeas
Since $n\geq k+1$, it follows that $\ol N(r,a_{1};F)\not=S(r,f)$. As $F-a_{1}$ and $F^{(k)}-a_{2}$ share $``(0,1)"$, it follows that $\ol N(r,a_{2};F^{(k)})\not=S(r,f)$.\\
Let $z_{0}$ be a common zero of $F-a_{1}$ and $F^{(k)}-a_{2}$ such that $a_{i}(z_{0})\not=0,\infty$ (otherwise the reduced counting functions of those zeros of $F-a_{1}$ and $F^{(k)}-a_{2}$ which are the zeros or poles of $a_{1}(z)$ and $a_{2}(z)$ respectively are equal to $S(r,f)$), where $i=1,2$. Clearly $F(z_{0}),\; F^{(k)}(z_{0})\not=0$.
Suppose $z_{0}$ is a zero of $F-a_{1}$ of multiplicity $p_{0}$. Since $F-a_{1}$ and $F^{(k)}-a_{2}$ share $``(0,1)"$, it follows that $z_{0}$ must be a zero of $F^{(k)}-a_{2}$ of multiplicity $q_{0}$.
Then in some neighbourhood of $z_{0}$, we get by Taylor's expansion
\beas F(z)=a_{10}+a_{1r_{0}}(z-z_{0})^{r_{0}}+a_{1r_{0}+1}(z-z_{0})^{r_{0}+1}+\ldots, a_{10}\not=0\eeas
\beas a_{1}(z)=b_{10}+b_{1s_{0}}(z-z_{0})^{s_{0}}+b_{1s_{0}+1}(z-z_{0})^{s_{0}+1}+\ldots, b_{10}\not=0.\eeas
Since $z_{0}$ is a zero of $F-a_{1}$ of multiplicity $p_{0}$, it follows that $a_{10}=b_{10}$ and $p_{0} \geq \min\{r_{0}, s_{0}\}$. 
Let us assume that
\beas F(z)-a_{1}(z)=c_{1p_{0}}(z-z_{0})^{p_{0}}+c_{1p_{0}+1}(z-z_{0})^{p_{0}+1}+\ldots, c_{1p_{0}}\not=0.\eeas
Therefore $\frac{F(z)-a_{1}(z)}{a_{1}(z)}=O((z-z_{0})^{p_{0}})$ and so $\frac{F(z)}{a_{1}(z)}-1=O((z-z_{0})^{p_{0}})$.
Similarly $\frac{F^{(k)}(z)-a_{2}(z)}{a_{2}(z)}=O((z-z_{0})^{q_{0}})$ and $\frac{F^{(k)}(z)}{a_{2}(z)}-1=O((z-z_{0})^{q_{0}})$.\\ Finally we conclude that 
$F-a_{1}$ and $F^{(k)}-a_{2}$ share $``(0,1)"$ if and only if $\frac{F}{a_{1}}$ and $\frac{F^{(k)}}{a_{2}}$ share $``(1, 1)"$ except for the zeros and poles of 
$a_{1}(z)$ and $a_{2}(z)$ respectively.\par
Let $F_{1}=\frac{f^{n}}{a_{1}}$ and  $G_{1}=\frac{(f^{n})^{(k)}}{a_{2}}$. Clearly $F_{1}$ and $G_{1}$ share $``(1, 1)"$ except for the zeros and poles of 
$a_{1}(z)$ and $a_{2}(z)$ respectively and so $\ol N(r,1;F_{1})=\ol N(r,1;G_{1})+S(r,f)$. 
Let \bea\label{e1} \Phi=\frac{F_{1}'(F_{1}-G_{1})}{F_{1}(F_{1}-1)} = \frac{F'_{1}}{F_{1} -1}\left(1 - \frac{G_{1}}{F_{1}}\right) = \frac{F_{1}'}{F_{1} - 1}\left(1 
- \frac{a_{1}}{a_{2}}\cdot \frac{F^{(k)}}{F}\right).\eea
We now consider the following two cases.\\
{\bf Case 1.} Let $\Phi\not\equiv 0$. Then clearly $G_{1}\not\equiv F_{1}$, i.e., $(f^{n})^{(k)}\not\equiv \frac{a_{2}}{a_{1}} f^{n}$.
Now from (\ref{e1}) we get
$m(r,\infty;\Phi)=S(r,f)$. 

Let $z_{1}$ be a zero of $f$ of multiplicity $p$ such that $a_{i}(z_{1})\not=0, \infty$, where $i=1,2$. Then $z_{1}$ will be a zero of $F_{1}$ and $G_{1}$ of 
multiplicities $np$ and $np-k$ respectively and so from (\ref{e1}) we get 
\be\label{e11a} \Phi(z)=O((z-z_{1})^{np-k-1}).\ee

Since $n\geq k+1$, it follows that $\Phi$ is holomorphic at $z_{1}$.

Let $z_{2}$ be a common zero of $F_{1}-1$ and $G_{1}-1$ such that $a_{i}(z_{2})\not=0, \infty$, where $i=1,2$.
Suppose $z_{2}$ is a zero of $F_{1}-1$ of multiplicity $q$. Since $F_{1}$ and $G_{1}$ share $``(1,1)"$ except for the zeros and poles of $a_{1}(z)$ and $a_{2}(z)$
respectively, it follows that $z_{2}$ must be a zero of $G_{1}-1$ of multiplicity $r$.
Then in some neighbourhood of $z_{2}$, we get by Taylor's expansion
\beas F_{1}(z)-1=b_{q}(z-z_{2})^{q}+b_{q+1}(z-z_{2})^{q+1}+\ldots, b_{q}\not=0\eeas
\beas G_{1}(z)-1=c_{r}(z-z_{2})^{r}+c_{r+1}(z-z_{2})^{r+1}+\ldots, c_{r}\not=0.\eeas

Clearly 
\beas F_{1}'(z)=qb_{q}(z-z_{2})^{q-1}+(q+1)b_{q+1}(z-z_{2})^{q}+\ldots.\eeas
Note that
\beas  F_{1}(z)-G_{1}(z)=\left\{\begin{array}{clcr} & b_{q}(z-z_{2})^{q}+\ldots,&\;\;\;\;\;\;\;\;{\text {if}}\; q<r \\   & -c_{r}(z-z_{2})^{r}-
\ldots,&\;\;\;\;\;\;\; {\text {if}}\; q>r\\   & (b_{q}-c_{q})(z-z_{2})^{q}+\ldots,&\;\;\;\;\;\;\; {\text {if}}\; q=r.\end{array}\right.\eeas 
Clearly from (\ref{e1}) we get \be\label{e11as} \Phi(z)=O\big((z-z_{2})^{t-1}\big),\ee where $t\geq\min\{ q, r\}$.
Now from (\ref{e11as}), it follows that $\Phi$ is holomorphic at $z_{2}$. \par
We note from (\ref{e1}) that if $z_{*}$ is a zero of $F_{1} - 1$ that is also a zero of $a_{2}$ with multiplicity $p_{1}$, then $z_{*}$ is a possible pole of 
$\Phi$ with multiplicity at most $1 + p_{1}$. Again if $z^{*}$ is a zero of $f$ that is also a zero of $a_{2}$ with multiplicity $p_{2}$, then $z^{*}$ is a 
possible pole of $\Phi$ with multiplicity at most $k + p_{2}$.
So from (\ref{e1}), above discussion and the hypothesis of \textrm{Theorem \ref{t1}} we note that \beas N(r, \infty; \Phi) &\leq & (k + 1)N(r, \frac
{a_{1}}{a_{2}}) + (k + 1)N(r,0; a_{1}) + (k + 1)N(r, 0; a_{2})\\ && + (k + 1)\ol N(r, F_{1}) + (k + 1)\ol N(r, f) \\ & = & (k + 1)\ol N(r, F_{1}) + S(r, f) \\ & 
= & S(r, f). \eeas
Consequently $T(r,\Phi)=S(r,f)$.\par
Let $q\geq 2$. Since $F_{1}$ and $G_{1}$ share $``(1,1)"$ except for the zeros and poles of $a_{1}(z)$ and $a_{2}(z)$, it follows that $r\geq 2$.
Therefore from (\ref{e11as}) we see that \beas \ol N_{(2}(r,1;F_{1})\leq N(r,0;\Phi) + S(r, f)\leq T(r,\Phi) + S(r, f) = S(r,f).\eeas 
 Since $F_{1}$ and $G_{1}$ share $``(1,1)"$ except for the zeros and poles of $a_{1}(z)$ and $a_{2}(z)$, it follows that $\ol N_{(2}(r,1;G_{1})=S(r,f)$.
Again from (\ref{e1}) we get \beas \frac{1}{F_{1}}=\frac{1}{\Phi}\left(\frac{F_{1}'}{F_{1} - 1} -\frac{F_{1}'}{F_{1}}\right)\Big[1-\frac{a_{1}}{a_{2}}\frac{(f^{n})^{(k)}}{f^{n}}\Big] \eeas and so $m(r,\frac{1}{F_{1}})=S(r,f)$. 
Hence \bea\label{e2} m(r, 0; f) = m(r,\frac{1}{f})=S(r,f).\eea 
We consider the following two sub-cases.\\
{\bf Sub-case 1.1.} Let $n>k+1$.\\
From (\ref{e11a}) we see that $N(r,0;f)\leq N(r,0;\Phi)\leq T(r,\Phi)+O(1)=S(r,f)$. Then
from (\ref{e2}) we get $T(r,f)=S(r,f)$, which is a contradiction.\\
{\bf Sub-case 1.2.} Let $n=k+1$.\\
Since for $p \geq 2$, we have $np - k - 1 = (k + 1)p - k - 1 \geq p$, from (\ref{e11a}) we see that $$N_{(2}(r,0;f)\leq N(r,0;\Phi)\leq T(r,\Phi)+O(1)=S(r,f).$$
Then (\ref{e2}) gives \bea\label{e14} T(r,f)=N_{1)}(r,0;f)+S(r,f).\eea
Note that $\ol N_{(2}(r, a_{1};F) = \ol N_{(2}(r,1; F_{1}) + S(r, f) = S(r,f)$, $\ol N_{(2}(r, a_{2};F^{(k)}) = \ol N_{(2}(r,1; G_{1}) + S(r, f) = S(r,f)$ and $\ol N(r,\infty;F) = S(r,f)$.  
Let \bea\label{e5} \beta=\frac{F^{(k)}-a_{2}}{F-a_{1}},\;\text{i.e.},\; F^{(k)}-a_{2}=\beta (F-a_{1}).\eea
We claim that $\beta\not\equiv 0$. If not, suppose $\beta\equiv 0$. Then from (\ref{e5}) we have $(f^{n})^{(k)}\equiv a_{2}$. Since $n=k+1$, we immediately have $N_{1)}(r,0,f)=S(r,f)$ and so from (\ref{e14}) we arrive at a contradiction. Hence $\beta\not\equiv 0$.
We now consider following two sub-cases.\\
{\bf Sub-case 1.2.1.} Suppose $T(r,\beta)\not=S(r,f)$.\\ 
Let $z_{11}$ be a zero of $F-a_{1}$ such that $F^{(k)}(z_{11})-a_{2}(z_{11})\not=0$. Then obviously $\beta$ has a pole at $z_{11}$. Let $z_{12}$ be a zero of $F^{(k)}-a_{2}$ such that $F(z_{12})-a_{1}(z_{12})\not=0$. In that case $\beta$ has a zero at $z_{12}$.
Let $z_{13}$ be a common zero of $F-a_{1}$ and $F^{(k)}-a_{2}$. Since $F-a_{1}$ and $F^{(k)}-a_{2}$ share $``(0,1)"$, it follows that $\beta$ has a zero at $z_{13}$ if $z_{13}$ is a zero of $F-a_{1}$ and $F^{(k)}-a_{2}$ with multiplicities $p_{13}(\geq 2)$ and $q_{13}(\geq 2)$ respectively such that $p_{13}<q_{13}$ and $\beta$ has a pole at $z_{13}$ if $q_{13}<p_{13}$.
Therefore \beas \ol N(r,0;\beta)\leq \ol N_{(2}(r,a_{2};F^{(k)})+S(r,f)=S(r,f)\eeas and 
\beas \ol N(r,\infty;\beta)\leq \ol N_{(2}(r,a_{1};F)+S(r,f)=S(r,f).\eeas
Let $\xi=\frac{\beta'}{\beta}$. Clearly
\beas T(r,\xi)= N(r,\infty;\frac{\beta'}{\beta})+ m(r,\frac{\beta'}{\beta})=\ol N(r,0;\beta)+ \ol N(r,\infty;\beta)+S(r,\beta)=S(r,f)+S(r,\beta).\eeas 
Note that
\beas  T(r,\beta)&\leq& T(r, F^{(k)}-a_{2})+T(r,F-a_{1})\\&\leq & T(r,F^{(k)})+T(r,F)+ S(r,F)+ S(r,G)\\&\leq& (k+1)T(r,f^{n})+ nT(r,f)+S(r,f)\\&=& n(k+2)T(r,f)+S(r,f),\eeas
which implies that $S(r,\beta)$ can be replaced by $S(r,f)$. Consequently $T(r,\xi)=S(r,f)$.
By logarithmic differentiation we get from (\ref{e5}) 

\bea\label{e7} F^{(k+1)}F-\xi F^{(k)}F-F^{(k)}F'&=&a_{1}F^{(k+1)}-(\xi a_{1}+a_{1}')F^{(k)}-a_{2}F'\\&&+(a_{2}'-\xi a_{2})F+\xi a_{1}a_{2}+a_{2}a_{1}'-a_{1}a_{2}'\nonumber.\eea
We deduce from (\ref{e4}) that 
\bea\label{e8} F^{(k)}&=&\frac{d^{k}}{dz^{k}}\Big\{f^{k+1}\Big\}\\&=& \frac{d^{k-1}}{dz^{k-1}}\Big\{(k+1)f^{k}f'\Big\}\nonumber\\&=& (k+1)\frac{d^{k-2}}{dz^{k-2}}\Big\{k f^{k-1}(f')^{2}+f^{k}f''\Big\}\nonumber\\&=& (k+1)k\;\frac{d^{k-2}}{dz^{k-2}}\Big\{f^{k-1}(f')^{2}\Big\} +(k+1)\frac{d^{k-2}}{dz^{k-2}}\Big\{f^{k}f''\Big\}\nonumber\\&=& (k+1)k\;\frac{d^{k-3}}{dz^{k-3}}\Big\{(k-1)f^{k-2}(f')^{3}\Big\}+(k+1)k\;\frac{d^{k-3}}{dz^{k-3}}\Big\{2f^{k-1}f'f''\big\}\nonumber\\&&+(k+1)\frac{d^{k-3}}{dz^{k-3}}\Big\{k f^{k-1}f'f''\Big\}+(k+1)\frac{d^{k-3}}{dz^{k-3}}\Big\{f^{k}f'''\Big\}\nonumber\\&=&(k+1)k(k-1)\frac{d^{k-3}}{dz^{k-3}}\Big\{f^{k-2}(f')^{3}\Big\}+2(k+1)k\frac{d^{k-3}}{dz^{k-3}}\Big\{f^{k-1}f'f''\big\}\nonumber\\&&+(k+1)k\frac{d^{k-3}}{dz^{k-3}}\Big\{f^{k-1}f'f''\Big\}+(k+1)\frac{d^{k-3}}{dz^{k-3}}\Big\{f^{k}f'''\Big\}\nonumber\\&=& \ldots\ldots\nonumber \\&=& (k+1)!f(f')^{k}+\frac{k(k-1)}{4}(k+1)!f^{2}(f')^{k-2}f''+\ldots +(k+1)f^{k}f^{(k)}\nonumber. \eea
Therefore \bea\label{e8sm} \frac{f'}{f}F^{(k)}&=&(k+1)!(f')^{k+1}+\frac{k(k-1)}{4}(k+1)!f(f')^{k-1}f''+\ldots+\\&&(k+1)f^{k-1}f'f^{(k)}\nonumber\eea
and 
\bea\label{e9} F^{(k+1)}&=&(k+1)!(f')^{k+1}+\frac{k(k+1)}{2}(k+1)!f(f')^{k-1}f''+\ldots+\\&&(k+1)f^{k}f^{(k+1)}\nonumber.\eea 
Substituting (\ref{e4}), (\ref{e8}), (\ref{e8sm}) and (\ref{e9}) into (\ref{e7}), we have \bea\label{e10} f^{n}(z)P(z)=Q(z),\eea where $Q(z)$ is a differential polynomial in $f$ of degree $n$ and 
\bea\label{e11} P(z)&=& F^{(k+1)}-\xi F^{(k)}-n\frac{f'}{f} F^{(k)}\\&=& -k(k+1)!(f')^{k+1}-(k+1)!\xi f(f')^{k}\nonumber\\&&+\frac{k(k+1)(3-k)(k+1)!}{4} f(f')^{k-1}f''+\ldots+(k+1)f^{k}f^{(k+1)}\nonumber\\&&-(k+1)\xi f^{k}f^{(k)}-(k+1)^{2}f^{k-1}f'f^{(k)}=-k(k+1)!(f')^{k+1}+R_{1}(f)\nonumber,\eea is a differential polynomial in $f$ of degree $k+1$, where $R_{1}(f)$ is a differential polynomial in $f$ such that each term of $R_{1}(f)$ contains $f^{m}$ for some  $m (1\leq m\leq n-1)$ as a factor.\par     

We suppose that $P\equiv 0$. Then from (\ref{e11}) we get $\displaystyle F^{(k + 1)} - \xi F^{(k)} - n\frac{f'}{f}F^{(k)} \equiv 0$ and so $\displaystyle 
\frac{F^{(k + 1)}}{F^{(k)}} = \xi + n\frac{f'}{f} = \frac{\beta'}{\beta} + \frac{F'}{F}.$
By integration we have $F^{(k)}=D\beta F$, where $D\in\mathbb{C}\setminus\{0\}$.  
Since $n=k+1$ and $\ol N(r,\infty;\beta)=S(r,f)$, it follows that $\ol N(r,0;f)=S(r,f)$.
Then from (\ref{e14}) we have $T(r,f)=S(r,f)$, which is a contradiction.
So $P\not\equiv 0$. Then by \textrm{Lemma \ref{l1}} we get $m(r,P)=S(r,f)$. Since $N(r,f) = S(r, f)$ we have \bea\label{e12} T(r,P)=S(r,f)\;\; \text{and} \;\; T(r,P')=S(r,f).\eea
Note that from (\ref{e11}) we get
\bea\label{e15} P'(z)= A_{1}(f')^{k}f''+B_{1}(f')^{k+1}+S_{1}(f),\eea is a differential polynomial in $f$, where $A_{1}=-\frac{1}{4}k(k+1)^{2}(k+1)!$, $B_{1}=-(k+1)!\xi $ and $S_{1}(f)$ is a differential polynomial in $f$ such that each term of $S_{1}(f)$ contains $f^{m}$ for some  $m(1\leq m\leq n-1)$ as a factor.\par     
Let $z_{3}$ be a simple zero of $f$ such that $\xi(z_{3})\not= 0, \infty $. Then from (\ref{e11}) and (\ref{e15}) we have \beas P(z_{3})=-k(k+1)!(f'(z_{3}))^{k+1},\;\; P'(z_{3})= A_{1}(f'(z_{3}))^{k}f''(z_{3})+B_{1}(z_{3})(f'(z_{3}))^{k+1}.\eeas
This shows that $z_{3}$ is a zero of $Pf''-[K_{1}P'-K_{2}P]f'$, where $K_{1}=\frac{-k(k+1)!}{A_{1}}$ and $K_{2}=\frac{B_{1}}{A_{1}}$. Also $T(r,K_{1})=S(r,f)$ and $T(r,K_{2})=S(r,f)$.
Let \bea\label{e16} \Phi_{1}=\frac{Pf''-[K_{1}P'-K_{2}P]f'}{f}.\eea 
Then clearly $m(r, \Phi_{1}) = S(r, f)$ and since $N_{(2}(r, 0; f) + N(r, f) = S(r, f)$, we have  $T(r,\Phi_{1})=S(r,f)$.
From (\ref{e16}) we obtain \bea\label{e17} f''(z)=\alpha_{1}(z)f(z)+\beta_{1}(z)f'(z),\eea where \bea\label{e18} \alpha_{1}=\frac{\Phi_{1}}{P}\;\text{and}\;\beta_{1}=K_{1}\frac{P'}{P}-K_{2}.\eea 

Differentiating (\ref{e17}) and using it repeatedly we have \bea\label{e17sm} f^{(i)}(z)=\alpha_{i-1}(z)f(z)+\beta_{i-1}(z)f'(z),\eea where $i\geq 2$
and $T(r,\alpha_{i-1})=S(r,f)$, $T(r,\beta_{i-1})=S(r,f)$.

Also (\ref{e18}) yields \bea\label{e19}P'=\Big(\frac{\beta_{1}}{K_{1}}+\frac{K_{2}}{K_{1}}\Big)P\eea and \beas \beta_{1}=K_{1}\frac{P'}{P}-K_{2}=\frac{-k(k+1)!}{A_{1}}\frac{P'}{P}-\frac{B_{1}}{A_{1}},\eeas so that \bea\label{e1911} A_{1}\beta_{1}+B_{1}+k(k+1)!\frac{P'}{P}=0.\eea
Now we consider following two sub-cases.\\
{\bf Sub-case 1.2.1.1.} Let $k=1$.\\
Now from (\ref{e11}) and (\ref{e17}) we have  \beas \label{s301}P=-2(f')^{2}-2\xi ff'+2ff''=-2(f')^{2}+(2\beta_{1}-2\xi)ff'+2\alpha_{1}f^{2}\eeas
and so \beas \label{s302}P'=(-2\beta_{1}-2\xi)(f')^{2}+(2\beta_{1}'-2\xi'+2\beta_{1}^{2}-2\beta_{1}\xi)ff'+(2\alpha_{1}\beta_{1}-2\alpha_{1}\xi+2\alpha_{1}')f^{2}.\eeas
Note that $K_{1}=1$ and $K_{2}=\xi$ and so from (\ref{e19}) we have \bea\label{s3022} \big(\beta_{1}'-\xi'-\beta_{1}\xi+\xi^{2}\big)f'+\big(-2\alpha_{1}\xi+\alpha_{1}'\big)f\equiv 0.\eea
If $-2\alpha_{1}\xi + \alpha_{1}' \equiv 0$, then from (\ref{s3022}) we get, because $f f' \not\equiv 0$, 
\bea\label{s3022a}\beta_{1}'-\xi'-\beta_{1}\xi+\xi^{2} \equiv 0.\eea 
Let $\beta_{1} \equiv \xi$. 
Then a simple calculation gives $\displaystyle 2\frac{\beta'}{\beta} =\frac{P'}{P}$ and so on integration we get $\displaystyle \beta^{2} = d_{0} P$, where $d_{0}\in\mathbb{C}\setminus\{0\}$. This contradicts the fact that $T(r, \beta ) \neq S(r, f)$. So $\beta_{1} \not \equiv \xi$. Now from   
(\ref{s3022a}) we get $\displaystyle \frac{\beta_{1}' - \xi'}{\beta_{1} - \xi} = \xi = \frac{\beta'}{\beta}$. So on integration we get $\beta = d_{1}(\beta_{1} - \xi)$, 
where $d_{1}\in\mathbb{C}\setminus\{0\}$. This contradicts the fact that $T(r, \beta) \neq S(r, f)$. 
So we conclude that $-2\alpha_{1}\xi+\alpha_{1}'\not\equiv 0$.

Then from (\ref{s3022}) we see that if $z_{4}$ is a simple zero of $f$, then $z_{4}$ is either a pole of $-2\alpha_{1}\xi + \alpha_{1}'$ or a zero of $\beta_{1}' - \xi' - \beta_{1}\xi + \xi^{2}$. Hence $$ N_{1)}(r, 0; f) \leq N(r, \infty; -2\alpha_{1}\xi + \alpha_{1}') + N(r, 0; \beta_{1}' - \xi' - \beta_{1}\xi + \xi^{2}) = S(r, f).$$ So we arrive at a contradiction by (\ref{e14}).\\ 
{\bf Sub-case 1.2.1.2.} Let $k\geq 2$. \\
From (\ref{e8}) and (\ref{e9}) we have $F^{(k)}=T_{1}(f)$, $F^{(k+1)}=(k+1)!(f')^{k+1}+T_{2}(f)$ and $F^{(k+2)}=\frac{(k+1)(k+2)}{2}(k+1)!(f')^{k}f''+T_{3}(f)$, where $T_{1}(f)$, $T_{2}(f)$ and $T_{3}(f)$ are differential polynomials in $f$ such that each term of $T_{1}(f)$, $T_{2}(f)$ and $T_{3}(f)$ contain $f$ as a factor.\\
Comparing (\ref{e7}) and (\ref{e10}) and noting that $F = f^{n} = f^{k + 1}$ we have \bea\label{s304}  Q &=& a_{1}F^{(k+1)}-\big(\xi a_{1}+a_{1}'\big)F^{(k)}-a_{2}F' +\big(a_{2}'-\xi a_{2}\big)F+\gamma\\&=& a_{1}\{(k+1)!(f')^{k+1}+ T_{2}(f)\}-(\xi a_{1}+a_{1}')T_{1}(f)-(k+1)a_{2}f^{k}f'\nonumber\\&&+(a_{2}'-\xi a_{2})f^{k+1}+\gamma\nonumber,\eea where $\gamma=\xi a_{1}a_{2}+a_{2}a_{1}'-a_{1}a_{2}'$. 

Now suppose $\gamma(z)\equiv 0$. Then by integration we obtain $\beta=d_{2}\frac{a_{2}}{a_{1}}$, where $d_{2}\in\mathbb{C}\setminus\{0\}$ and so $T(r,\beta)=S(r,f)$, which is a contradiction. 
Consequently $\gamma(z)\not\equiv 0$. Similarly we can verify that $\xi a_{1}+a_{1}'\not\equiv 0$ and $a_{2}' - \xi a_{2} \not\equiv 0$. We further note that $T(r,\gamma)= S(r,f)$.
Differentiating (\ref{s304}) we have \bea\label{s305} \;\;\;\;Q' &=& a_{1}'F^{(k+1)}+a_{1}F^{(k+2)}-(\xi a_{1}+a_{1}')F^{(k+1)}-(\xi a_{1}+a_{1}')'F^{(k)}-a_{2}'F'-a_{2}F''\\&&+(a_{2}'-\xi a_{2})'F+(a_{2}'-\xi a_{2}) F'+\gamma'\nonumber\\&=& a_{1}'\Big\{(k+1)!(f')^{k+1}+T_{2}(f)\Big\}+a_{1}\Big\{\frac{(k+1)(k+2)}{2}(k+1)!(f')^{k}f''+T_{3}(f)\Big\}\nonumber\\&&-(\xi a_{1}+a_{1}')\Big\{(k+1)!(f')^{k+1}+T_{2}(f)\Big\}-(\xi a_{1}+a_{1}')'T_{1}(f)-(k+1)a_{2}'f^{k}f'\nonumber\\&&-a_{2}\Big\{k(k+1)f^{k-1}(f')^{2}+(k+1)f^{k}f''\Big\}+(a_{2}'-\xi a_{2})'f^{k+1}\nonumber\\&& +(k+1)(a_{2}'-\xi a_{2})f^{k}f'+\gamma'\nonumber.\eea

Let $z_{5}$ be a simple zero of $f(z)$ such that $z_{5}$ is not a zero or a pole of $a_{1}$, $a_{2}$ and $\xi$. Then from (\ref{e10}), (\ref{s304}) and (\ref{s305}) we have 
\beas \gamma(z_{5})=A(z_{5})(f'(z_{5}))^{k+1},\;\;
\gamma'(z_{5})=A_{2}(z_{5})(f'(z_{5}))^{k}f''(z_{5})+B_{2}(z_{5})(f'(z_{5}))^{k+1},\eeas
where $A(z)=-(k+1)!a_{1}(z)$, $A_{2}(z)= -\frac{(k+1)(k+2)}{2}(k+1)!a_{1}(z)$ and $B_{2}(z)=(k+1)!\xi(z)a_{1}(z)$.
This shows that $z_{5}$ is a zero of $\gamma f''-[K_{3}\gamma'-K_{4}\gamma]f'$, where $K_{3}=\frac{A}{A_{2}}$ and  $K_{4}=\frac{B_{2}}{A_{2}}$. Also $T(r,K_{3})=S(r,f)$ and $T(r,K_{4})=S(r,f)$.
 
Let \bea\label{s306} \Phi_{2}=\frac{\gamma f''-[K_{3}\gamma'-K_{4}\gamma]f'}{f}.\eea 
Then clearly $T(r,\Phi_{2})=S(r,f)$.
From (\ref{s306}) we obtain \bea\label{s307} f''=\phi_{1}f+\psi_{1}f',\eea where \bea\label{s308} \phi_{1}=\frac{\Phi_{2}}{\gamma}\;\text{and}\;\psi_{1}=K_{3}\frac{\gamma'}{\gamma}-K_{4}.\eea

Now we show that $\psi_{1}\not\equiv\beta_{1}$. If $\psi_{1}\equiv\beta_{1}$ then from (\ref{e18}) and (\ref{s308}) we have
\beas \frac{2}{(k+1)(k+2)}\frac{\gamma'}{\gamma}+\frac{2}{(k+1)(k+2)}\xi \equiv \frac{4}{(k+1)^{2}}\frac{P'}{P}-\frac{4}{k(k+1)^{2}}\xi,\eeas 
i.e., \beas 2k(k+2)\frac{P'}{P}-k(k+1)\frac{\gamma'}{\gamma}\equiv (k^{2}+3k+4)\frac{\beta'}{\beta}.\eeas

On integration we have \beas \beta^{k^{2}+3k+4}\equiv \frac{d_{3} P^{2k(k+2)}}{\gamma^{k(k+1)}},\eeas where $d_{3}\in\mathbb{C}\setminus\{0\}$
and so from (\ref{e12}) we have $T(r,\beta)=S(r,f)$, a contradiction. 

Now from (\ref{s307}) we have \bea\label{s309} f^{(i)}=\phi_{i-1}f+\psi_{i-1}f',\eea where $i\geq 2$
and $T(r,\phi_{i-1})=S(r,f)$, $T(r,\psi_{i-1})=S(r,f)$.
Also from (\ref{e11}), (\ref{e15}) and (\ref{s309}) we have respectively \bea\label{s310} P=-k(k+1)!(f')^{k+1}+\sum\limits_{j=1}^{k+1}T_{j}f^{j}(f')^{k+1-j},\eea 
\bea\label{s312} &&P'=(A_{1}\psi_{1}+B_{1})(f')^{k+1}+\sum\limits_{j=1}^{k+1}S_{j}f^{j}(f')^{k+1-j},\eea where $T(r,T_{j})=S(r,f)$ and $T(r,S_{j})=S(r,f)$.

Multiplying (\ref{s310}) by $P'$ and (\ref{s312}) by $P$ and then subtracting we get \bea\label{e211} H_{0}(f')^{k+1}+H_{1}f(f')^{k}+\ldots+H_{k+1}f^{k+1}\equiv 0,\eea
where \bea\label{e1912} H_{0}=P\Big[A_{1}\psi_{1}+B_{1}+k(k+1)!\frac{P'}{P}\Big]\eea and 
$H_{j}=PS_{j}-P'T_{j}$ for $j=1,2,\ldots, k+1$. 
Since $\beta_{1}\not\equiv \psi_{1}$ and $P\not\equiv 0$, it follows from (\ref{e1911}) and (\ref{e1912}) that $H_{0}\not\equiv 0$.
Again since $H_{0}(f')^{k+1}\not\equiv 0$, from (\ref{e211}) we conclude that $H_{i}\not\equiv 0$ for at least one $i\in\{1,2,\ldots,k+1\}$.
Let $S=\{1,2,\ldots, k+1\}$ and $S_{1}=\{i\in S: H_{i}\not\equiv 0\}$. Note that $T(r,H_{0})=S(r,f)$ and $T(r,H_{j})=S(r,f)$ for $j\in S_{1}$.\\
Now from (\ref{e211}) we see that a simple zero of $f$ must be either a zero of $H_{0}$ or a pole of at least one $H_{i}$'s, where $i\in S_{1}$.
Therefore \[N_{1)}(r,0;f) \leq N(r, 0; H_{0})+\sum\limits_{\substack{j\\j\in S_{1}}} N(r,\infty;H_{j}) + S(r, f) = S(r,f).\] 
So we arrive at a contradiction by (\ref{e14}).\\
{\bf Sub-case 1.2.2.} Suppose $T(r,\beta)=S(r,f)$.\\
Then from (\ref{e5}) we have \bea\label{e23a} F^{(k)}-\beta F\equiv a_{2}-\beta a_{1}.\eea
If $a_{2}-\beta a_{1}\equiv 0$, then from (\ref{e23a}) we get $(f^{n})^{(k)}\equiv \frac{a_{2}}{a_{1}} f^{n}$, which contadicts the fact that $\Phi\not\equiv 0$. So we suppose that
$a_{2}-\beta a_{1}\not\equiv 0$. Let $z_{6}$ be a simple zero of $f$. If $z_{6}$ is not a pole of $\beta$, then from (\ref{e23a}) we see that $z_{6}$ is a zero of $a_{2} - a_{1}\beta$. 
Therefore 
\[N_{1)}(r, 0; f) \leq N(r, 0; a_{2} - a_{1}\beta) + N(r, \infty ; \beta) = S(r, f).\]
So by (\ref{e14}) we arrive at a contradiction.\\
{\bf Case 2.} Let $\Phi\equiv 0$. Now from (\ref{e1}) we get $F_{1}\equiv G_{1}$, i.e., $ (f^{n})^{(k)}\equiv \frac{a_{2}}{a_{1}}f^{n}$.

Furthermore if $a_{1} \equiv a_{2}$, then $f^{n}\equiv (f^{n})^{(k)}$, and by \textrm{Lemma \ref{l9}}, $f$ assumes the form $f(z)=ce^{\frac{\lambda}{n}z}$, where $c\in\mathbb{C}\setminus\{0\}$ and $\lambda^{k}=1$.\end{proof}

\begin{proof}[Proof of Theorem \ref{t2}] 
Let $F_{1}=\frac{f^{n}}{a_{1}}$ and  $G_{1}=\frac{(f^{n})^{(k)}}{a_{2}}$. Clearly $F_{1}$ and $G_{1}$ share $``(1, 0)"$ except for the zeros and poles of 
$a_{1}(z)$ and $a_{2}(z)$ and so $\ol N(r,1;F_{1})=\ol N(r,1;G_{1})+S(r,f)$. 
We now consider following two cases.\\
{\bf Case 1.} Let $F_{1}\not \equiv G_{1}$.\\
Then \bea\label{e24} \ol N(r,1;F_{1})& \leq& \ol N(r,0;G_{1}-F_{1}\mid F_{1}\neq 0)+S(r,f)\\&\leq & \ol N(r,0;\frac{G_{1}-F_{1}}{F_{1}})+S(r,f)
\nonumber\\&\leq& T(r,\frac{G_{1}-F_{1}}{F_{1}})+S(r,f)\nonumber\\&\leq& T(r,\frac{G_{1}}{F_{1}})+S(r,f)\nonumber\\&\leq&  N(r,\infty;\frac{G_{1}}{F_{1}})+m(r,
\infty;\frac{G_{1}}{F_{1}})+S(r,f)\nonumber\\&=& N(r,\infty;\frac{a_{1}}{a_{2}}\frac{(f^{n})^{(k)}}{f^{n}})+m(r,\infty;\frac{a_{1}}{a_{2}}\frac
{(f^{n})^{(k)}}{f^{n}})+S(r,f)\nonumber\\ 
 &\leq &k\;\ol N(r,\infty;f)+ k\;\ol N(r,0;f^{n})+S(r,f) \nonumber \\ &=& k\;\ol N(r,0;f)+S(r,f).\nonumber\eea 
Now using (\ref{e24}) and $\ol N_{2)}(r, 0; f) = S(r, f)$, we get from the second fundamental theorem that \bea\label{e25} n\;T(r,f)&=& T(r,f^{n})+S(r,f)\\&\leq& 
T(r,F_{1})+S(r,f)\nonumber\\ &\leq & \ol 
N(r,\infty;F_{1})+\ol N(r,0;F_{1})+\ol N(r,1;F_{1})+S(r,F)\nonumber\\ &\leq & \;\ol N(r,\infty;f)+\ol N(r,0;f^{n})+\ol N(r,1;F_{1})+S(r,f)\nonumber\\ &\leq &(k+1)
\;\ol N(r,0;f) + S(r,f) \nonumber \\
& \leq & \frac{k + 1}{3} N(r, 0; f) + S(r, f) \nonumber \\
&\leq & \frac{k + 1}{3} T(r, f) + S(r, f).\nonumber\eea 
Since $n\geq k$, (\ref{e25}) leads to a contradiction.\\
{\bf Case 2.} $F_{1} \equiv G_{1}$. Then
$ (f^{n})^{(k)}\equiv \frac{a_{2}}{a_{1}}f^{n}$.
Furthermore if $a_{1} \equiv a_{2}$, then $f^{n}\equiv (f^{n})^{(k)}$, and by \textrm{Lemma \ref{l9}}, $f$ assumes the form $f(z)=ce^{\frac{\lambda}{n}z}$, where $c\in\mathbb{C}\setminus\{0\}$ and $\lambda^{k}=1$. \end{proof}

\end{document}